\documentclass[11pt,twoside]{preprint}
\usepackage[a4paper]{geometry}

\usepackage{amsmath,amsthm,amssymb,enumerate}
\usepackage{hyperref}
\usepackage{mathrsfs}
\usepackage{breakurl}
\usepackage[nameinlink,capitalize]{cleveref}

\usepackage[osf]{newtxtext}
\usepackage[varqu,varl]{inconsolata}
\usepackage[bigdelims,vvarbb,cmintegrals]{newtxmath}

\usepackage{xcolor}
\usepackage{mhequ}
\usepackage{comment}
\usepackage{microtype}
\usepackage{pifont}% http://ctan.org/pkg/pifont

\usepackage{algpseudocode}                                          %added by T for pseudocode
\usepackage{graphicx}                                                    %added by T for figures
\usepackage{subcaption}                                                 %added by T also for figures
\usepackage[section]{placeins}                                         % added by T to force figures to appear in the section they were declared in

\colorlet{darkblue}{blue!90!black}
\colorlet{darkred}{red!90!black}
\colorlet{darkgreen}{green!60!black}

\newtheorem{theorem}{Theorem}[section]
\newtheorem{lemma}[theorem]{Lemma}

\newtheorem{corollary}[theorem]{Corollary}

\theoremstyle{definition}
\newtheorem{assumption}[theorem]{Assumption}

\theoremstyle{remark}
\newtheorem{remark}[theorem]{Remark}
\Crefname{assumption}{Assumption}{Assumptions}

\newcommand{\vn}[1]{{\vert\kern-0.23ex\vert\kern-0.23ex\vert #1 
    \vert\kern-0.23ex\vert\kern-0.23ex\vert}}
\usepackage{mathtools}

\DeclarePairedDelimiter\floor{\lfloor}{\rfloor}

\allowdisplaybreaks

\def\eps{\varepsilon}

\def\path{\mathrm{path}}

\def\({\left(}
\def\){\right)}

\begin{document}

\title{A note on the weak rate of convergence for the Euler-Maruyama scheme with Hölder drift}
\author{Teodor Holland}
%\institute{\textsterling  and \texteuro}
\maketitle

\begin{abstract}
We consider SDEs with bounded and $\alpha$-Hölder continuous drift, with $\alpha \in (0,1)$, driven by multiplicative noise. We show that under sufficient conditions on the diffusion matrix, which guarantee the existence of a unique strong solution, the weak rate of convergence for the Euler-Maruyama scheme is almost $(1+\alpha)/2$.
The present paper forms part of the author's master's thesis.
\end{abstract}
%\keywords{}
%\MSC{	60H10, 60H50,	65C30}
\tableofcontents

%%%%%%%%%%%%%%%%%%%%%%%%%%%%%%%%%%%%%%%%%%%%%%%%%%%%%%%%%%%%%%%%%%%%%%%%%%%%%%%%%%%
\section{Introduction}
Consider a complete probability space $(\Omega, \mathscr{F},\mathbb{P})$ carrying a $d$-dimensional Wiener process $(W_t)_{t\in [0,1]}$ for arbitrary $d \in \mathbb{Z}^+$ . Let $b:\mathbb{R}^d \rightarrow \mathbb{R}^d$ be bounded and $\alpha$-Hölder continuous for some $\alpha \in (0,1)$. Let $\sigma:\mathbb{R}^d \rightarrow \mathbb{R}^{d\times d}$ be bounded, with bounded and Lipschitz-continuous first partial derivatives. We will furthermore assume that the diffusion matrix satisfies the uniform ellipticity condition that $\sigma \sigma^T \geq \lambda I$ in the sense of positive definite matrices, for some constant $\lambda>0$.
For times $t \in [0,1]$ we will consider the SDE
\begin{equ}[oursde]
dX_t = b(X_t)dt + \sigma(X_t)dW_t, \quad X_0 = x_0.
\end{equ}
It is known that while the deterministic counterpart of (\ref{oursde}) is not well-posed, the presence of the noise turns it into a well posed problem (see: \cite{zvonkin1974transformation}, \cite{veretennikov1981strong}, \cite{davie2007uniqueness}).

We will moreover consider the Euler-Maruyama scheme 
with identical initial condition
\begin{equ}[ourscheme]
dX_t^n = b(X_{\kappa_n(t)}^n)dt + \sigma(X_{\kappa_n(t)}^n)dW_t, \quad X_0^n = x_0,
\end{equ}
where $\kappa_n(t) := \frac{\lfloor nt \rfloor}{n}$ and $\floor{\cdot}$ denotes integer part.

\subsection{Literature}
Maruyama's stochastic generalisation of the Euler-scheme is the simplest and most commonly used numerical method for approximating solutions of SDEs.
The convergence of the Euler-Maruyama scheme for the smooth coefficient case has been studied extensively. It is well-known that for bounded and Lipschitz coefficients the strong rate is $1/2$, and if the coefficients are four times continuously differentiable with bounded partial derivatives, then the weak rate is $1$ (see e.g. \cite{pages2018numerical}). The study of the weak error in the smooth setting has been pioneered by Talay and Tubaro (see e.g. \cite{talay1990expansion}) who derived weak error expansions for the scheme.

The case of irregular drift is significanly less well-understood and has attracted much attention in the past years. Below we give an overview of some of the recent results.

Even though it has been shown in 1980 by Veretennikov that the SDE (\ref{oursde}) admits a unique strong solution when $b$ is merely bounded and measurable, it has only been shown in 1996 by Gyöngy and Krylov (see \cite{gyongy1996existence} and \cite{gyongy2021existence}) that in the same setting the Euler-Maruyama scheme converges in probability. 

Much of the research on the convergence of the scheme is concerned with the strong rate.
In \cite{pamen2017strong} Pamen and Taguchi consider SDEs with additive noise  and $\alpha$-Hölder drift with $\alpha \in (0,1)$, and they show strong rate of order $\alpha/2$. They moreover derive a similar result for when the driving noise is a truncated symmetric $\tilde{\alpha}$-stable process with $\tilde{\alpha} \in (1,2)$.
The case of piecewise Lipschitz drift has been studied in \cite{leobacher2018convergence} by Leobacher and Szölgyenyi, where they achieve a strong rate of  $1/4$.
In \cite{bao2019convergence} Bao, Huang, and Yuan derive strong convergence results in the nondegenerate case where the drift is Dini and unbounded, and in the degenerate case under Hölder-Dini and Dini and conditions on the drift.
In $\cite{mikulevivcius2018rate}$, Mikulevi{\v{c}}ius and Xu study the strong error of SDEs driven by multiplicative $\tilde{\alpha}$-stable processes with $\tilde{\alpha} \in [1,2)$ and with Hölder drift.
In \cite{muller2020performance} Müller-Gronbach and Yaroslavtseva consider SDEs in a $1$-dimension with piecewise Lipschitz drift and multiplicative noise, and show convergence of order $1/2$.
In \cite{butkovsky2021approximation} Butkovsky, Dareiotis and Gerencsér consider SDEs with $\alpha$-Hölder drift, driven by additive fractional Brownian noise with Hurst parameter $H \in (0,1)$. They show convergence almost of order $(\frac{1}{2}+\alpha H)\wedge 1$, provided that  $\alpha> 1-1/(2H)$.
In \cite{dareiotis2021quantifying} rate $1/2$ is achieved for bounded and measurable drift.
In \cite{le2021taming} Lê and Ling study the tamed Euler-Maruyama scheme for drift satisfying the so-called Ladyzhenskaya–Prodi–Serrin condition.

Due to its relevance in applications, the weak error of the Euler-Maruyama scheme with irregular coefficients has also attraced recent interest.
Recall the classic result by Mikulevičius and Platen in \cite{mikulevicius1991rate} that in our setting the weak rate should be at least ${\alpha/2}$. However this rate deteriorates for small $\alpha$. 
In  \cite{kohatsu2011weak},  Kohatsu-Higa, Lejay and Yasuda consider the convergence of a tamed Euler-Maruyama scheme.
A possible way to study the weak convergence of the scheme is through weak error expansions. For recent work in this direction, we point the reader to e.g. \cite{huang2015stable}, and the references within.
In \cite{konakov2017weak} Konakov and Menozzi consider SDEs with explicitly time-dependent coefficients. In the case when the drift is $\alpha$-Hölder in space and $\alpha/2$-Hölder in time, they achieve weak rate $\alpha/2$.
For the case of $L^Q - L^\rho$-drift and additive noise, we point the reader to the recent work \cite{jourdain2021convergence} of Jourdain and Menozzi, where they show weak convergence of rate $\frac{1}{2}(1- (d/\rho+2/q))$.

Even from the recent results on strong approximation (see \cite{butkovsky2021approximation}, \cite{dareiotis2021quantifying}) it follows that in our setting the weak rate $\alpha/2$ cannot be optimal, since even the strong rate is $1/2$.
In fact, in \cite{dareiotis2021quantifying} it is shown that in the additive case if the drift is bounded and has regularity of order $\alpha$, then the strong rate is $(1+\alpha)/2$.
It is thus reasonable to conjecture that rate of $(1+\alpha)/2$ could be achieved for the weak error in the multiplicative noise case. This has been noted by Jourdain and Menozzi in \cite{jourdain2021convergence}.

\subsection{On the method of proof}\label{onthemethodofproofsection}
The proof of the main result uses the idea of using the Feynmann-Kac framework to express the weak error as a comparison between the initial- and terminal-time values of the solution of a parabolic PDE evaluated at the approximation process in space. According to \cite{bally1996law} this method was introduced by Talay  in \cite{talay1984efficient} and \cite{talay1986discretisation}, and independently by Milstein in \cite{mil1986weak}.
Using this method, we reduce the problem to having to prove  quadrature estimates. Showing that the desired quadrature estimates hold requires the use of stochastic sewing, which was developed by L{\^e} in \cite{le2020stochastic} and has been used in the context of regularization by noise for the Euler-Maruyama scheme e.g. in \cite{butkovsky2021approximation} and its weighted version in \cite{dareiotis2021quantifying} for proving similar bounds.
\subsection{Notation}
In the remainder of the introduction, we will introduce some standard notation, that will be used throughout the paper.

\subsubsection{Hölder-spaces}
Let $A\subseteq \mathbb{R}^d$ and $(B,|\cdot|)$ a normed space.
For $\alpha \in (0,1]$ and $f:A \to B$ we define the $\alpha$-Hölder seminorm of $f$ by
$$[f]_{C^\alpha(A,B)} : = \sup_{\substack{x,y \in A\\ x\neq y}} \frac{|f(x) - f(y)|}{|x-y|^\alpha}.$$
For $\alpha \in (0, \infty)$ we then denote by $C^\alpha(A,B)$ the space of all functions such that for all $l \in (\mathbb{Z}^+)^d$ multiindices with $|l|< \alpha$, the derivative $\partial^l f$ exists, and 
\begin{equ}
\| f\|_{C^\alpha(A,B)} := \sum_{|l|< \alpha} \sup_{x \in A} | \partial^l f(x) | + \sum_{\alpha-1<|l|<\alpha} [\partial^l f]_{C^{\alpha - |l|}(A,B)} <\infty.
\end{equ}
We furthermore define $C^0(A,B)$ to be the space of all measurable functions $f:A \to B$ such that
\begin{equ}
\|f\|_{C^0(A,B)} := \sup_{x \in A} |f(x)| < \infty.
\end{equ}
The definition of Hölder-spaces extends to negative $\alpha$ as follows: for $\alpha \in (-\infty,0)$ we say that the distribution $f$ is of class $C^\alpha(\mathbb{R}^d, \mathbb{R})$, if
\begin{equs}
\|f\|_{C^\alpha(A,B)} : = \sup_{\varepsilon \in (0,1]} \varepsilon^{-\alpha/2}\| \mathcal{P}_\varepsilon f \|_{C^0(\mathbb{R}^d, \mathbb{R})} <\infty,
\end{equs}
where $\mathcal{P}_\varepsilon f := p_\varepsilon * f$ denotes convolution with the standard heat kernel (\ref{heatkernel}).

When no ambiguity can arise, we will simply write $C^\alpha(A)$ or $C^\alpha$ to mean $C^\alpha(A,B)$. 

\begin{remark}
Using Hölder-spaces, the regularity assumptions on our coefficients can be stated simply as follows: $b \in C^\alpha$ for some $\alpha \in (0,1)$, and $\sigma \in C^2$.
\end{remark}

\subsubsection{Wasserstein metric}

We will see that our weak convergence result about the scheme $X_t^n$ can be rephrased as a convergence result about the probability measure associated with the scheme. A natural way of comparing probability densities is using the so-called Wasserstein metric, which is defined as follows:
Let $(M,d)$ be a metric space. We define the Wasserstein (also known as Kantorovich-Rubinstein) norm of a measure $\mu:\mathscr{B}(M) \to [0, \infty]$ by
\begin{equ}
\|\mu\|_{\text{Wass}} : = \sup\left\{ \int_M g(x) \mu(dx) : \|g\|_{C^1} \leq 1\right\}.
\end{equ}
We furthermore denote the norm induced by this metric by
\begin{equ}
d_{\text{Wass}}(\mu_1,\mu_2) : = \|\mu_1 - \mu_2\|_{\text{Wass}}.
\end{equ}

\subsubsection{Shorthands}
Consider the filtration generated by $W$, that is $\mathscr{F}^W_t := \sigma(W_s; s\leq t)$ and its augmentation

 $\mathscr{F}_t := \sigma(\mathscr{F}_t \cup \mathscr{N})$ by the collection of null sets
 $\mathscr{N} := \{N \subseteq \Omega : N\subseteq F \text{ for some } F \subseteq \mathscr{F}_\infty^W \text{ such that } \mathbb{P}(F) =0\}$. For $t \in [0,1]$ we will write $\mathbb{E}^t(\cdot)$ as a shorthand for the conditional expectation $\mathbb{E}(\cdot|\mathscr{F}_t )$.

Moreover, in proofs, we  will write $f \lesssim g$ to mean that $f \leq N g$ for some $N$ constant where the dependence of $N$ is as specified in the statement we are proving.

\section{Main Results}

The main result of this paper is as follows:

\begin{theorem}\label{mainresult}
Let $\alpha \in (0,1)$, $g \in C^\alpha$, and fix $\varepsilon>0$. Let $X$ be the solution of (\ref{oursde}) and $X^n$ be given by (\ref{ourscheme}). Then for all $n \in \mathbb{N}$, we have
\begin{equation}
\left|\mathbb{E}g(X_1) - \mathbb{E}g(X_1^n)\right| \leq N n^{-\frac{1+\alpha}{2} + \varepsilon}
\end{equation}
where $N$ is a constant depending only on $d,\varepsilon,\alpha, \lambda,\|b\|_{C^\alpha}, \|\sigma\|_{C^2}, \|g\|_{C^\alpha}$.
\end{theorem}

This in particular implies that the distribution of $X_1^n$ converges at the same rate to the distribution of $X_1$ in the Wasserstein metric:

\begin{theorem}
Let $(X_t)_{t \in [0,1]}$ be the solution of (\ref{oursde}) and let $(X_t^n)_{t \in [0,1]}$ be given by (\ref{ourscheme})
Let $\mu_{X_1}$ and $\mu_{X_1^n}$ be the distributions of $X_1$ and $X_1^n$ respectively and fix $\varepsilon>0$. Then for all $n \in \mathbb{N}$, we have
\begin{equ}
d_{\text{\normalfont{Wass}}}(\mu_{X_1}, \mu_{X_1^n}) \leq Nn^{-\frac{1+\alpha}{2}+\varepsilon}
\end{equ}
where $N$  is a constant depending only on $d, \varepsilon,\alpha,\lambda,\|b\|_{C^\alpha}, \|\sigma\|_{C^2}$.
\end{theorem}

\section{Reduction to quadrature estimates}\label{reductionsection}
For motivation, we will begin by giving a sketch of our strategy of proof for the weak error bound, which will be presented in detail in section \ref{proofofmainresultsection}. Using the Feynmann-Kac formula, the weak error can be written as
$$d_g(X,X^n) = \left|\mathbb{E}\left(u(1,X_1^n) -u(0,X_0^n)\right)\right|$$
where $u$ is the (unique bounded) solution of a parabolic PDE
\begin{equ}\label{ourpde}
\begin{cases}
\partial_t u(t,x) + Lu(t,x) = 0 &\forall(t,x) \in [0,1)\times \mathbb{R}^d \\
u(1,x) = g(x)  &\forall x \in \mathbb{R}^d
\end{cases}
\end{equ}
where $L$ is the generator of the solution of (\ref{oursde}).
It can be shown that
$$d_g(X,X^n)= \left|\mathbb{E}\int_0^1\left( \bar{L} u \left(r, X_r^n, X_{\kappa_n(r)}^n\right)  - Lu(r,X_r^n)\right)\,dr\right|,$$
where the operator $\bar{L}$ is the ``frozen'' generator associated with our SDE (\ref{oursde}), that is:
$$\bar{L}\phi(t,x, \bar{x}) : = \sum_{i=1}^d b(\bar{x})\partial_{x_i} \phi(x) + \frac{1}{2}\sum_{i,j=1}^d (\sigma\sigma^T)_{ij}(\bar{x})\partial_{x_i x_j}\phi(x).$$
Writing out the operators $L$ and $\bar{L}$ explicitly, and using the triangle inequality gives that
\begin{multline*}
d_g(X^n,X)\lesssim \sup_{i\in \{1,\dots,d\}}\left|\mathbb{E}\int_0^1 \left(\left(b_i(X_{\kappa_n(r)}^n) - b_i(X_r^n)\right)\partial_{x_i}u(r,X_r^n)\right)\,dr\right|\\
+  \sup_{i,j \in \{1,\dots,d\}}\left|\mathbb{E}\int_0^1 \left(\left(\sigma\sigma^T)_{ij}(X_{\kappa_n(r)}^n) - (\sigma\sigma^T)_{ij}(X_r^n)\right)\partial_{x_i x_j}u(r,X_r^n)\right)\,dr\right|.
\end{multline*}
Therefore in order to show that $d_g(X^n,X)$ converges with the desired rate, it suffices to show that the same rate of convergence holds for
\begin{equation}\label{bexpression0}
\left|\mathbb{E}\int_0^1 \left(\left(b_i(X_{\kappa_n(r)}^n) - b_i(X_r^n)\right)\partial_{x_i}u(r,X_r^n)\right)\,dr\right|
\end{equation}
and
\begin{equation}\label{sigmaexpression0}
\left|\mathbb{E}\int_0^1 \left(\left(\sigma\sigma^T)_{ij}(X_{\kappa_n(r)}^n) - (\sigma\sigma^T)_{ij}(X_r^n)\right)\partial_{x_i x_j}u(r,X_r^n)\right)\,dr\right|,
\end{equation}
for all $i,j$ indices. 

This way the proof of the main result is reduced to having to show that the expressions (\ref{bexpression0}) and (\ref{sigmaexpression0}) converge at the desired rate. In order to prove this, in section \ref{usefulestimatessection} we will prove estimates on $u$ and we will state a lemma which will allow us to exploit the regularising property of convolution with the density of the Euler-Maruyama scheme. We will then introduce stochastic sewing in section \ref{sewingsection}, which will be necessary to bound (\ref{bexpression0}). Using these results, in section \ref{quadratureestimatessection} we will prove quadrature estimates on (\ref{bexpression0}) and (\ref{sigmaexpression0}). Finally, in section \ref{proofofmainresultsection} we prove the main result.

\section{Some useful estimates}\label{usefulestimatessection}
In this section, we will introduce some estimates that will be used throughout the proof of the main result.
We will begin by recalling some basic estimates on the parabolic PDE associated with the expectation $\mathbb{E}g(X_1)$ in the Feynmann-Kac framework. This estimate will be derived from a more general, well-known estimate on parabolic PDEs, that we will state below.
Consider the parabolic PDE
\begin{equ}\label{ourpde}
\begin{cases}
\partial_t u(t,x) + Lu(t,x) = 0 &\forall(t,x) \in [0,1)\times \mathbb{R}^d \\
u(1,x) = g(x)  &\forall x \in \mathbb{R}^d
\end{cases}
\end{equ}
where the differential operator $L$ is the infinitesimal generator of $X$, that is, for smooth $\phi:\mathbb{R}^d \to\mathbb{R}$ it is defined as
\begin{equ}
L\phi(x) := \frac{1}{2}\sum_{i,j=1}^d (\sigma(x)\sigma^T(x))_{ij}\partial_{x_i x_j}\phi(x)
+\sum_{i=1}^d b_i(x)\partial_{x_i}\phi(x).
\end{equ}
For the proof of the main result we will need estimates on the solution of (\ref{ourpde}). These can be obtained by using a Schauder-type estimate on the parabolic PDE corresponding to the differential operator $\mathscr{L}$ that is defined on smooth functions $\phi:\mathbb{R} \to \mathbb{R}^d$ by
\begin{equ}\label{parabolicoperator}
\mathscr{L}\phi(x) = \sum_{i,j=1}^da_{ij}(x)\partial_{x_i x_j}\phi(x) + \sum_{i=1}^d b_i(x)\partial_{x_i}\phi(x) + c(x)\phi(x),
\end{equ}
where the coefficients satisfy the following assumption:
\begin{assumption}\label{parabolicassumptions}{\ } \\
\vspace*{-0.8cm}
\begin{enumerate}[(i)] 
\item There exist constants $K>0$ and $\alpha \in (0,1)$ such that for all $i,j \in \{1,\dots,d\}$ we have $\|a_{ij}\|_{C^{\alpha}},\|b_{j}\|_{C^{\alpha}},\|c\|_{C^{\alpha}}<K$
\item for $A(x) = (a_{ij}(x))_{1\leq i,j \leq d}$ there exists a positive constant $\mu$ such that $\langle A(x)\xi,\xi\rangle \geq \mu|\xi|^2$ for all $x,\xi \in \mathbb{R}^d$.
\end{enumerate}
\end{assumption}

Let $\mathcal{C}^{1,2}$ denote the class of bounded continuous functions $u:[0,1]\times \mathbb{R}^d \to \mathbb{R}$, such that $\partial_t u, \nabla u, \nabla^2 u$ exist and are continuous on $(0,1) \times \mathbb{R}^d$.
The following lemma can be found as Proposition 6.4.1 and Theorem 6.4.3 in \cite{lorenzi2021semigroups}.
\begin{lemma}[A Schauder-type estimate for parabolic PDEs]\label{schauder}\label{parabolicestimates}
Suppose that the operator (\ref{parabolicoperator}) satisfies Assumption \ref{parabolicassumptions}.
For each bounded and continuous initial condition $g$, the Cauchy problem
\begin{equ}
\begin{cases}
\partial_t v(t,x) = \mathscr{L}v(t,x),\quad &\forall (t,x) \in(0,1) \times \mathbb{R}^d,\\
v(0,x) = g(x), & \forall x \in \mathbb{R}^d
\end{cases}
\end{equ}
admits a unique solution $v$ in the class $\mathcal{C}^{1,2}$.

Moreover, for each $0\leq \beta\leq\gamma\leq2+\alpha$ there exists two positive constants $N_1$ and $N_2$ depending only on $\beta,\gamma,\alpha, d,K, \lambda$ such that
$$\|v(t,\cdot)\|_{C^\gamma(\mathbb{R}^d) }\leq N_1 t^{ - \frac{\gamma - \beta}{2}}e^{N_2t}\|g\|_{C^\beta(\mathbb{R}^d)}.$$
\end{lemma}

The above lemma can be directly applied to the time reversal of the solution of (\ref{ourpde}) to obtain the following estimate.

\begin{corollary}[A Schauder-type estimate for $u$]\label{uestimates}
Let $u$ be the unique bounded solution of (\ref{ourpde}). Then for each $0\leq \beta\leq\gamma\leq2+\alpha$  there exists two positive constants $N_1$ and $N_2$ depending only on $\| b\|_{C^\alpha}, \|\sigma\|_{C^\alpha}, \alpha, d, \lambda$  such that for all $t \in (0,1)$ we have
\begin{equ}
\|u(t,\cdot)\|_{C^\gamma(\mathbb{R}^d)}\leq N_1(1-t)^{-\frac{\gamma - \beta}{2}}e^{N_2(1-t)}\|g\|_{C^\beta(\mathbb{R}^d)}.
\end{equ}
\end{corollary}

The following lemma (which has been proven via Malliavin-calculus in \cite{butkovsky2021approximation}) will be used in order to exploit the regularising property of convolution with the density of the Euler-Maruyama scheme.
Consider the driftless scheme
\begin{equ}\label{driftless}
\bar{X}_t^n = \sigma(\bar{X}_{\kappa_n(t)}^n)dW_t, \quad \bar{X}_0^n = y \in \mathbb{R}^d.
\end{equ}
\begin{lemma}\label{G'ofdriftless}
Let $\sigma \in C^2$, with $\sigma\sigma^T \geq \lambda I$ for some constant $\lambda>0$, let $\bar{X}^n$ be given by (\ref{driftless}) and let $G \in C^1$. Then for all $t =1/n, 2/n, \dots 2$ and $k=1,\dots d$ one has the bound
\begin{equ}
|\mathbb{E}\partial_{x_k}G(\bar{X}_t^n)|\leq N\|G\|_{C^0} t^{-1/2} + N \|G\|_{C^1} e^{-cn}
\end{equ}
with some constants $N = N(d,\lambda, \|\sigma\|_{C^2})$ and $c = c(d,\|\sigma\|_{C^2})>0$.
\end{lemma}

\section{Sewing}\label{sewingsection}

For $0\leq S \leq T\leq 1$ we define the set $[S,T]_{\leq} := \{ (s,t) : S\leq s\leq t\leq T\}$. Moreover for  a function $A_{\cdot,\cdot}$ defined on $[S,T]_{\leq}$ and for times $s\leq u\leq t$ we set $\delta A_{s,u,t} := A_{s,t} - A_{s,u} - A_{u,t}$.

Lê's stochastic sewing lemma which has been proven in \cite{le2020stochastic} is as follows:

\begin{lemma}[Stochastic sewing lemma]\label{sewinglemma}
Let $p\geq0$, $0\leq S\leq T\leq 1$ and let $A_{.,.}$ be a function $[S,T]_\leq \rightarrow L_p(\Omega, \mathbb{R}^d)$ such that for any $(s,t) \in [S,T]_\leq$ the random vector $A_{s,t}$ is $\mathscr{F}_t$-measurable. Suppose that for some $\varepsilon_1, \varepsilon_2 >0$ and $C_1,C_2$ the bounds
\begin{itemize}
\item[\normalfont{(S1)}] $\|A_{s,t}\|_{L_p(\Omega)}\leq C_1 |t-s|^{1/2 + \varepsilon_1}$
\item[\normalfont{(S2)}] $\|\mathbb{E}^s \delta A_{s,u,t}\|_{L_p(\Omega)}\leq C_2 |t-s|^{1+\varepsilon_2}$
\end{itemize}
hold for all $S\leq s\leq u\leq t\leq T$. Then there exists a unique map $\mathcal{A}:[S,T]\rightarrow L_p(\Omega, \mathbb{R}^d)$ such that $\mathcal{A}_S =0$, $\mathcal{A}_t$ is $\mathscr{F}_t$-measurable for all $t \in [S,T]$, and the following bounds hold for some constants $K_1,K_2>0$:
\begin{itemize}
\item[\normalfont{(S3)}] $\|\mathcal{A}_t - \mathcal{A}_s - A_{s,t}\|_{L_p(\Omega)} \leq K_1 |t-s|^{1/2 + \varepsilon_1}, \quad  \forall (s,t) \in [S,T]_{\leq}$
\item[\normalfont{(S4)}] $\| \mathbb{E}^s(\mathcal{A}_t  - \mathcal{A}_s - A_{s,t})\|_{L_p(\Omega)}\leq K_2|t-s|^{1+\varepsilon_2}, \quad \forall (s,t) \in [S,T]_{\leq}.$
\end{itemize}
Moreover, there exists a constant $K$ depending only of $\varepsilon_1, \varepsilon_2, d, p$ such that $\mathcal{A}$ satisfies the bound
$$\|\mathcal{A}_t - \mathcal{A}_s\|_{L_p(\Omega)}\leq KC_1 |t-s|^{1/2 +\varepsilon_1} + KC_2|t-s|^{1+\varepsilon_2}, \quad \forall (s,t) \in [S,T]_{\leq}.$$
\end{lemma}

\section{Quadrature estimates}\label{quadratureestimatessection}

We will begin by introducing some notation.
Let $\bar{X}^n$ be the driftless scheme given by (\ref{driftless}), and let $u$ be the solution of the parabolic PDE (\ref{ourpde}). Fix some $i \in \{1,\dots,d\}$
and for convenience, let us denote 
$$f_t := \partial_{x_i} u(t, \cdot).$$
Let furthermore $p_t:\mathbb{R}^d \to \mathbb{R}$ denote the standard heat kernel given by
\begin{equ}\label{heatkernel}
p_t(x) := \frac{1}{(2\pi t)^{d/2}}\exp\left(-\frac{|x|^2}{2t}\right)
\end{equ}
for $x \in \mathbb{R}^d$ and $t> 0$, and by $p_0(x)$ we denote the Dirac-delta. 

More generally, for a positive semi-definite matrix $M \in \mathbb{R}^{d\times d}$, we denote by $p_{M}:\mathbb{R}^d \to \mathbb{R}$ the Gaussian density given by
\begin{equ}
p_{M}(x) := \frac{1}{(\det(2\pi M))^{1/2}}\exp\left( -\frac{1}{2} x^T M^{-1} x \right)
\end{equ}
which coincides with $p_t$ whenever $M = t I$ where $I$ denotes the identity matrix.

We will now give an outline of this section. Recall that according to the reduction sketched in section \ref{reductionsection}, in order to prove the main result, we will need to show that
the expressions
\begin{equation}\label{bexpression1}
\left|\mathbb{E}\int_0^1 \left(\left(b_i(X_{\kappa_n(r)}^n) - b_i(X_r^n)\right)\partial_{x_i}u(r,X_r^n)\right)\,dr\right|
\end{equation}
and
\begin{equation}\label{sigmaexpression1}
\left|\mathbb{E}\int_0^1 \left(\left(\sigma\sigma^T)_{ij}(X_{\kappa_n(r)}^n) - (\sigma\sigma^T)_{ij}(X_r^n)\right)\partial_{x_i x_j}u(r,X_r^n)\right)\,dr\right|
\end{equation}
converge at the desired rate of almost $(1+\alpha)/2$ for all $i,j$ indices. 

Due to the poor regularity of the drift $b$, the main challange of the proof is bounding (\ref{bexpression1}). 
Note that doing this by using only the regularity of $b$, we would obtain a bound that deteriorates for small $\alpha$. We will thus need to also exploit the regularising property of convolution with the density of the Euler-Maruyama scheme
\footnote{For intuition, we may consider this to be analogous to using the regularity of the functional 
$x \mapsto \int_0^1 b(x+W_r)dr,$
rather than that of the function $b$.}.

Note however that the bound of Lemma \ref{G'ofdriftless} only holds when the drift coefficient of the Euler-Maruyama scheme is zero. Hence we will begin by showing that the desired bound holds when in (\ref{bexpression1}), the Euler-Maruyama process $X^n$ is replaced with the driftless scheme $\bar{X}^n$ (which is given by (\ref{driftless})) and the integral is taken over $[s,t]$ where $0\leq s\leq t\leq T <1$. (The latter condition needs to be introduced due to the blowup of $f_r = \partial_{x_i} u(t, \cdot)$ in time, as described in the Schauder-estimate (\ref{uestimates}).)

Note that the bound on (\ref{bexpression1}) with $\bar{X}^n$ in place of $X^n$ needs to be shown in the $L_p$-norm, so that it will imply (via a Girsanov-type transformation of drift argument) that the same bound holds when later the drift is re-introduced, i.e when $\bar{X}^n$ is replaced with $X^n$. In order to prove our $L_p$-bound on the expression featuring $\bar{X}^n$, we employ stochastic sewing in the spirit of the proofs of Lemma 6.1 in \cite{butkovsky2021approximation} and Lemma 3.3 in \cite{dareiotis2021quantifying}.

\begin{lemma}\label{bquadrature}
 For $p\geq1$, $\varepsilon>0$.  For all $n \in \mathbb{N}$,  $(s,t) \in [0,T]_{\leq}$, we have
\begin{multline*}
\left\|\int_s^t (b_i(\bar{X}_r^n)  - b_i(\bar{X}_{\kappa_n(r)}^n))f_r(\bar{X}_r^n)\,dr\right\|_{L_p} \leq N n^{-\frac{1+\alpha}{2} + \varepsilon}\Big(|1-T|^{-1/2}|t-s|^{1/2+\varepsilon} \\
+ |1-T|^{-1}|t-s|^{1+\varepsilon}\Big)
\end{multline*}
for some constant $N$ depending only on $d, \alpha,p, \varepsilon, \lambda, \|b\|_{C^\alpha}, \|\sigma\|_{C^2}, \|g\|_{C^0}$.
\end{lemma}

\begin{proof}
We will prove this using Lemma \ref{sewinglemma}. Notice that by Hölder's inequality it suffices to show the bound for $p\geq2$, furthermore  by a standard mollification argument if suffices to consider smooth $b$.

Suppose that $0\leq s\leq t\leq T <1$, and let

$$A_{s,t} : = \mathbb{E}^s \int_s^t (b_i(\bar{X}_r^n) - b_i(\bar{X}_{\kappa_n(r)}^n))f_r(\bar{X}_s^n)dr.$$
We will first check that condition (S1) holds. This will be done analogously to the proof of Lemma 6.1 in \cite{butkovsky2021approximation}. Note, that by using Corollary \ref{uestimates},
\begin{align}\label{Astfirstbound}
\| A_{s,t}\|_{L_p} &= \left\|\int_s^t f_r(\bar{X}_s^n) \mathbb{E}^s(b_i(\bar{X}_r^n) - b_i(\bar{X}_{\kappa_n(r)}^n)) \,dr\right\|_{L_p} \nonumber\\
&\lesssim |1-T|^{-1/2}\|g\|_{C^0} \int_s^t \| \mathbb{E}^s(b_i(\bar{X}_r^n) - b_i(\bar{X}_{\kappa_n(r)}^n))\|_{L_p}\,dr.
\end{align}
We will now proceed to bound the integrand. Suppose that $\kappa_n(s) = \frac{k}{n}$. We will first consider the case when $r\geq\frac{k+4}{n}$. Then we have $\kappa_n(r) \geq s$, and thus
\begin{align*}
\mathbb{E}^s (b_i(\bar{X}_r) - b_i(\bar{X}_{\kappa_n(r)}^n)) 
&= \mathbb{E}^s \mathbb{E}^{\kappa_n(r)}(b_i(\bar{X}_{\kappa_n(r)}^n + \sigma(\bar{X}_{\kappa_n(r)}^n)(W_r - W_{\kappa_n(r)})) - b_i(\bar{X}_{\kappa_n(r)}^n))\\
& =\mathbb{E}^s\left(\mathbb{E}\left(b_i(x + \sigma(x)(W_r - W_{\kappa_n(r)})) - b_i(x)\right)\Big\vert_{x = \bar{X}_{\kappa_n(r)}^n}\right).
\end{align*}
Hence defining $\tilde{g}:\mathbb{R}^d \to \mathbb{R}$ by
$$\tilde{g}(x) = \tilde{g}_r^n(x) : = \int_{\mathbb{R}}(b_i(x+y) - b_i(x))p_{\sigma\sigma^T(x)\delta}(y)\,dy$$
with $\delta : = r - \kappa_n(r) \leq n^{-1}$, and using that the restriction of $\bar{X}^n$ to the gridpoints $\frac{1}{n}, \frac{2}{n}, \dots 1$ is a Markov process, we have
\begin{align*}
\mathbb{E}^s (b_i(\bar{X}_r) - b_i(\bar{X}_{\kappa_n(r)}^n))  &= \mathbb{E}^s\tilde{g}(\bar{X}_{\kappa_n(r)}^n)\\
& = \mathbb{E}^s\mathbb{E}^{\frac{k+1}{n}}\tilde{g}\left(\bar{X}_{\kappa_n(r) - \frac{k+1}{n} + \frac{k+1}{n}}\right)\\
& = \mathbb{E}^s \left(\mathbb{E}\tilde{g}\left(\bar{X}_{\kappa_n(r) - \frac{k+1}{n}}(x)\right)\Big\vert_{x = \bar{X}_{\frac{k+1}{n}}^n}\right)
\end{align*}
and thus
\begin{equation}\label{integrandfirstbound}
|\mathbb{E}^s(b_i(\bar{X}_r^n) - b_i(\bar{X}_{\kappa_n(r)}^n))| \leq \sup_{x\in \mathbb{R}}|\mathbb{E}\tilde{g}(\bar{X}_{\kappa_n(r) - \frac{k+1}{n}}(x))|.
\end{equation}
Just like in the proof of Lemma 6.1 in \cite{butkovsky2021approximation}, we have
\begin{equation}\label{gtildenorm}
\|\tilde{g}\|_{C^{\alpha/2}} \lesssim \|b\|_{C^{\alpha}},
\end{equation}
and for the solution $\tilde{u}$ of $(1+\Delta)\tilde{u} = \tilde{g}$, we have
\begin{equation}\label{utildenorm}
\|\tilde{u}\|_{C^{2+\alpha/2}}\lesssim \|\tilde{g}\|_{C^{\alpha/2}}  \qquad\text{and}\qquad \|\tilde{u}\|_{C^{1+2\varepsilon}} \lesssim \|\tilde{g}\|_{C^{-1 + 2\varepsilon}},
\end{equation}
and thus using (\ref{integrandfirstbound}), Lemma \ref{G'ofdriftless},(\ref{gtildenorm}),(\ref{utildenorm}), we have
\begin{align}\label{integrandsecondbound}
\|\mathbb{E}^s(b_i(\bar{X}_r^n) - b_i(\bar{X}_{\kappa_n(r)}^n))\|_{L_p} &\leq \sup_{x \in \mathbb{R}^d}|\mathbb{E}(\tilde{u}+\Delta \tilde{u})(\bar{X}_{\kappa_n(r) - \frac{k+1}{n}}(x))| \nonumber\\
&\lesssim \|\tilde{u}\|_{C^1}|\kappa_n(r) - (k+1)/n|^{-1/2} + \|\tilde{u}\|_{C^2}e^{-cn} \nonumber\\
&\lesssim \|\tilde{u}\|_{C^{1+2\varepsilon}}|\kappa_n(r) - (k+1)/n|^{-1/2} + \|\tilde{u}\|_{C^{2+\alpha/2}}e^{-cn} \nonumber\\
&\lesssim \|\tilde{g}\|_{C^{-1 + 2\varepsilon}} |\kappa_n(r) - (k+1)/n|^{-1/2} + \|\tilde{g}\|_{C^{\alpha/2}}e^{-cn}.
\end{align}
We now need to bound $\|\tilde{g}\|_{C^{-1 + 2\varepsilon}}$. For convenience, we will use the notation $M(x,z): =\sigma\sigma^T(x-z)$. Note, that for $\gamma \in (0,1]$ we have
\begin{equs}
\mathcal{P}_\gamma \tilde{g}(x) &= \int_{\mathbb{R}^d} p_\gamma(z)\tilde{g}(x-z)\,dz\\
 &= \int_{\mathbb{R}^d} p_\gamma(z) \int_{\mathbb{R}^d}\left(b_i(x-z+y) - b_i(x-z)\right)p_{M(x,z)\delta}(y)\,dy\,dz\\
 & = - \sum_{j=1}^d \int_{\mathbb{R}^d}\int_{\mathbb{R}^d} p_\gamma(z)p_{M(x,z)\delta}(y)\int_0^1 y_j \partial_{z_j}b_i(x-z+\theta y)\, d\theta \,dy \,dz\\
 & = \sum_{j=1}^d \int_{\mathbb{R}^d}\int_{\mathbb{R}^d} \partial_{z_j}\left(p_\gamma(z)p_{M(x,z)\delta}(y)\right)y_j \int_0^1 b_i(x-z+\theta y)\,d\theta \,dy \,dz\\
& = \sum_{j=1}^d\left( I_1^j + I_2^j\right)\\
\end{equs}
where
\begin{equ}
I_1^j :=  \int_{\mathbb{R}^d}\int_{\mathbb{R}^d} \partial_{z_j}\left(p_\gamma(z)\right)p_{M(x,z)\delta}(y)y_j \int_0^1 b_i(x-z+\theta y)\,d\theta \,dy\, dz
\end{equ}
and
\begin{equ}
I_2^j := \int_{\mathbb{R}^d}\int_{\mathbb{R}^d} p_\gamma(z) \partial_{z_j}\left(p_{M(x,z)\delta}(y)\right)y_j \int_0^1 b_i(x-z+\theta y)\,d\theta \,dy \,dz.
\end{equ}
As when bounding these, the $j$-dependence will not play a role, we will drop the superscript $j$ and write $I_1$ and $I_2$ respectively.
First, we will bound $I_1$. Note that if we removed the $y$-dependency of $b_i$ in $I_1$, then the integral would be zero, that is
\begin{equs}
\tilde{I}_1 &:= \int_{\mathbb{R}^d}\int_{\mathbb{R}^d} \partial_{z_j}\left(p_\gamma(z)\right)p_{M(x,z)\delta}(y)y_j \int_0^1 b_i(x-z)\,d\theta \,dy \,dz\\
& = \int_{\mathbb{R}^d} \partial_{z_j}(p_\gamma(z))b_i(x -z) \left(\int_{\mathbb{R}^d} y_j p_{M(x,z)\delta}(y)\,dy \right) \,dz\\
&=0.
\end{equs}
Hence $I_1$ can be bounded as follows:
\begin{equs}
|I_1| &= |I_1 - \tilde{I}_1|\\
&= \left|\int_{\mathbb{R}^d}\int_{\mathbb{R}^d} \partial_{z_j}\left(p_\gamma(z)\right)p_{M(x,z)\delta}(y)y_j \int_0^1 (b_i(x-z+\theta y) - b_i(x-z))\,d\theta \,dy \,dz \right|\\
&\leq \int_{\mathbb{R}^d}\int_{\mathbb{R}^d} |\partial_{z_j}(p_\gamma(z))|p_{M(x,z)\delta}(y)|y_j| \|b\|_{C^\alpha}|y|^\alpha \,dy\,dz\\
&\lesssim \int_{\mathbb{R}^d}\int_{\mathbb{R}^d} \frac{|z|}{\gamma}p_\gamma(z)p_{M(x,z)\delta}(y)|y|^{1+\alpha}\,dy\,dz\\
&\lesssim \frac{1}{\gamma} n^{-\frac{1+\alpha}{2}}\gamma^{1/2}\\
& \lesssim \gamma^{-1/2}n^{-\frac{1+\alpha}{2}}.
\end{equs}
Here we used that $|\partial_{z_j} p_\gamma(z)|\lesssim |z| \gamma^{-1} p_\gamma(z)$ (see e.g. the proof of Lemma 6.1 in \cite{butkovsky2021approximation}) and that for $Z \sim \mathcal{N}(0,M(x,z)\delta)$, we have $\mathbb{E}|Z|^{1+\alpha} \lesssim \delta^{(1+\alpha)/2}$.

We will now bound $I_2$. To this end, note that
\begin{align}\label{heatkernelderivative}
&\partial_{z_j}(p_{M(x,z)\delta}(y))  \nonumber\\
&= \partial_{z_j}\left(\det(2\pi M(x,z)\delta)^{-1/2}) \exp\left( -\frac{1}{2} y^T(M(x,z)\delta)^{-1} y\right)\right) \nonumber\\
& = -\frac{1}{2}\det(2\pi M(x,z)\delta)^{-3/2}\partial_{z_j}(\det(2\pi M(x,z)\delta))\exp\left( -\frac{1}{2} y^T(M(x,z)\delta)^{-1} y\right) \nonumber\\
&\qquad + \det(2\pi M(x,z)\delta^{-1/2})\exp\left( -\frac{1}{2} y^T(M(x,z)\delta)^{-1} y\right) \partial_{z_j}\left(-\frac{1}{2}y^T(M(x,z)\delta)^{-1} y\right) \nonumber\\
& = -\frac{1}{2}\left(\frac{\partial_{z_j}(\det(M(x,z)))}{\det(M(x,z))} + 
\frac{\partial_{z_j}(y^T M(x,z)^{-1} y)}{\delta}\right)p_{M(x,z)\delta}(y)
\end{align}
where due to our assumptions on $\sigma$ the following bounds hold:
\begin{equs}\label{heatkernelbounds}
\left|\frac{\partial_{z_j}(\det(M(x,z)))}{\det(M(x,z))}\right| \lesssim 1 ,\qquad  \left|\frac{\partial_{z_j}(y^T M(x,z)^{-1} y)}{\delta}\right| \lesssim \delta^{-1}|y|^2.
\end{equs}
Using (\ref{heatkernelderivative}), $I_2$ can be decomposed as follows:
\begin{equ}
I_2 = I_{2,1} + I_{2,2}
\end{equ}
for
\begin{equ}
I_{2,1} =\int_{\mathbb{R}^d}\int_{\mathbb{R}^d} p_\gamma(z) 
\left(-\frac{1}{2}\frac{\partial_{z_j}(\det(M(x,z)))}{\det(M(x,z))}p_{M(x,z)\delta}(y)\right)
y_j \int_0^1 b_i(x-z+\theta y)\,d\theta\, dy \,dz
\end{equ}
and
\begin{equ}
I_{2,2} =  \int_{\mathbb{R}^d}\int_{\mathbb{R}^d} p_\gamma(z) 
\left( -\frac{1}{2}\frac{\partial_{z_j}(y^T M(x,z)^{-1} y)}{\delta}p_{M(x,z)\delta}(y)\right)
y_j \int_0^1 b_i(x-z+\theta y)\,d\theta \,dy\, dz.
\end{equ}
The expression $I_{2,1}$ is easier to bound. By the same reasoning as we used for $I_1$, we have that if we removed the $y$-dependence of $b_i$ in $I_{2,1}$ then the integral would be zero. Hence we have
\begin{align*}
|I_{2,1}|&= \Big|\int_{\mathbb{R}^d}\int_{\mathbb{R}^d} p_\gamma(z) 
\left(-\frac{1}{2}\frac{\partial_{z_j}(\det(M(x,z)))}{\det(M(x,z))}p_{M(x,z)\delta}(y)\right)
y_j \\
&\qquad\qquad\qquad\qquad\int_0^1 (b_i(x-z+\theta y) - b_i(x-z))\,d\theta \,dy \,dz\Big|\\
&\lesssim  \int_{\mathbb{R}^d}\int_{\mathbb{R}^d} p_\gamma(z) p_{M(x,z)\delta}(y)\|b\|_{C^\alpha}|y|^\alpha |y_j| \,dy\,dz\\
&\lesssim n^{-\frac{1+\alpha}{2}},
\end{align*}
where we used (\ref{heatkernelbounds}).

Finally, we need to bound $I_{2,2}$. It is less straightforward to see that if we removed the $y$-dependency of $b_i$ in $I_{2,2}$, the integral would be zero. To show this, define
\begin{align*}
\tilde{I}_{2,2} &=  \int_{\mathbb{R}^d}\int_{\mathbb{R}^d} p_\gamma(z) 
\left( -\frac{1}{2}\frac{\partial_{z_j}(y^T M(x,z)^{-1} y)}{\delta}p_{M(x,z)\delta}(y)\right)
y_j \int_0^1 b_i(x-z)\,d\theta \,dy \,dz\\
& =\int_{\mathbb{R}^d} -\frac{1}{2} p_\gamma(z)b_i(x-z)\frac{1}{\delta}
\left(\int_{\mathbb{R}^d}p_{M(x,z)\delta}(y) y_j \partial_{z_j}(y^T M(x,z)^{-1} y) \,dy\right)\,dz.
\end{align*}
We will show that in this expression the integral with respect to $y$ is always zero. To this end, note that
\begin{equs}
y_j \partial_{z_j}(y^T M(x,z)^{-1} y) &= y_j \partial_{z_j}\sum_{k=1}^d y_k (M(x,z)^{-1}y)_k\\
& = y_j \partial_{z_j}\sum_{k=1}^d y_k \sum_{l=1}^d (M(x,z)^{-1})_{kl}y_l\\
& = \sum_{k,l =1}^d y_j y_k y_l \partial_{z_j}(M(x,z)^{-1})_{kl},
\end{equs}
and therefore
\begin{equs}
\int_{\mathbb{R}}^d y_j \partial_{z_j}(y^T M(x,z)^{-1} y) p_{M(x,z)\delta}(y) \,dy = 
\sum_{l,k =1}^d\partial_{z_j}(M(x,z)^{-1})_{kl} \int_{\mathbb{R}^d}y_j y_k y_l  p_{M(x,z)\delta}(y)\,dy.
\end{equs}
To see that the integral in the right hand side expression is zero, note that
$$\int_{\mathbb{R}}^d y_j y_k y_l p_{M(x,z)\delta} dy = \mathbb{E}\left(Y_j Y_k Y_l\right)$$
with $Y \sim \mathcal{N}(0, M(x,z)\delta)$. Since $Y$ is a Gaussian random vector with mean zero, so is
$$\tilde{Y} :=(Y_j,Y_k,Y_l),$$
hence $-\tilde{Y}$ has the same distribution as $\tilde{Y}$. Thus with $q:\mathbb{R}^3 \to \mathbb{R}$, $p(x_1,x_2,x_3) =x_1x_2x_3$ we have
$$\mathbb{E}\left(Y_j Y_k Y_l\right) = \mathbb{E}q(\tilde{Y}) =\mathbb{E}q(-\tilde{Y}) = \mathbb{E}\left((-Y_j)(-Y_k)(-Y_l)\right) = -\mathbb{E}q(\tilde{Y}).$$
So since $\mathbb{E}q(\tilde{Y})= -\mathbb{E}q(\tilde{Y})$, we have
$$\mathbb{E}\left(Y_j Y_k Y_l\right) = \mathbb{E}q(\tilde{Y}) =0.$$
We have shown that $\tilde{I}_{2,2} =0$. Thus we have
\begin{align*}
&|I_{2,2}| \\
&= |I_{2,2} - \tilde{I}_{2,2}|\\
 &= \left| \int_{\mathbb{R}^d}\int_{\mathbb{R}^d} p_\gamma(z) 
\left( -\frac{1}{2}\frac{\partial_{z_j}(y^T M(x,z)^{-1} y)}{\delta}p_{M(x,z)\delta}(y)\right)
y_j \int_0^1 (b_i(x-z+\theta y) - b_i(x-z))\,d\theta \,dy \,dz\right|\\
&\leq \int_{\mathbb{R}^d}\int_{\mathbb{R}^d}p_\gamma(z)\left|\frac{\partial_{z_j}(y^T M(x,z)^{-1} y)}{\delta} \right|
p_{M(x,z)\delta}(y)|y_j|\|b\|_{C^\alpha}|y|^\alpha\, dy\,dz\\
&\lesssim \int_{\mathbb{R}^d}\int_{\mathbb{R}^d}p_{\gamma}(z)\delta^{-1}|y|^2 p_{M(x,z)\delta}(y)|y|^{1+\alpha}\,dy\,dz\\
& \lesssim \delta^{-1} \int_{\mathbb{R}^d} p_\gamma(z)\left( \int_{\mathbb{R}^d} |y|^{3+\alpha} p_{M(x,z)\delta}(y)\,dy \right)\,dz\\
&\lesssim \delta^{-1} \delta^{\frac{3+\alpha}{2}}\\
& \lesssim \delta^{\frac{1+\alpha}{2}}\\
& \lesssim n^{-\frac{1+\alpha}{2}},
\end{align*}
where we again used (\ref{heatkernelbounds}).

From the bounds on $|I_1|, |I_{2,1}|, |I_{2,2}|$, we conclude that
\begin{equ}
\|\mathcal{P}_\gamma \tilde{g}\|_{C^0} \lesssim \gamma^{-1/2}n^{-\frac{1+\alpha}{2}}.
\end{equ}
Note, that we also have the trivial estimate
\begin{equ}
\|\mathcal{P}_\gamma \tilde{g}\|_{C^0} \lesssim \|\tilde{g}\|_{C^0} \lesssim \|b\|_{C^0}.
\end{equ}
Using these two estimates, we have that for $\beta = -1+2\varepsilon$,
\begin{align*}
\|\tilde{g}\|_{C^\beta} &=\sup_{\gamma \in (0,1]} \gamma^{-\beta/2}\|\mathcal{P}_\gamma \tilde{g}\|_{C^0}\\
&\lesssim \sup_{\gamma \in (0,1]} \gamma^{-\beta/2}((\gamma^{-1/2}n^{-\frac{1+\alpha}{2}})\wedge \|b\|_{C^0} )\\
& \lesssim \sup_{\gamma \in (0,1]}\gamma^{-\beta/2}((\gamma^{-1/2}n^{-\frac{1+\alpha}{2}})\wedge 1 )\\
&\lesssim \sup_{\gamma \in (0,n^{-(1+\alpha)}]} \gamma^{-\beta/2}\cdot1 + \sup_{\gamma \in (n^{-(1+\alpha)},1]} \gamma^{-\beta/2 - 1/2}n^{-\frac{1+\alpha}{2}}\\
& \lesssim {\left(n^{-(1+\alpha)}\right)}^{1/2 - \varepsilon} + {\left(n^{-(1+\alpha)}\right)}^{-\varepsilon} n^{-\frac{1+\alpha}{2}}\\
&  \lesssim n^{-\frac{1+\alpha}{2} + (1+\alpha)\varepsilon}.
\end{align*}

Using this to bound the first term of (\ref{integrandsecondbound}) and using (\ref{gtildenorm}) to bound the second term of (\ref{integrandsecondbound}), we find that
\begin{equation}\label{integrandgoodbound}
\|\mathbb{E}^s (b_i(\bar{X}_r^n) - b_i(\bar{X}_{\kappa_n(r)}^n))\|_{L_p} \lesssim n^{-\frac{1+\alpha}{2} + (1+\alpha)\varepsilon} |\kappa_n(r) - (k+1)/n|^{-1/2}.
\end{equation}
Recall that to obtain this bound, we assumed that $r\geq \frac{k+4}{n}$. We also have the trivial bound
\begin{equation}\label{integrandtrivialbound}
\|\mathbb{E}^s(b_i(\bar{X}_r^n) - b_i(\bar{X}_{\kappa_n(r)}^n))\|_{L_p} \lesssim n^{-\alpha/2}
\end{equation}
which holds for any $r\geq s$.

We are now ready to bound $\| A_{s,t}\|_{L_p}$. First suppose that $t \geq \frac{k+4}{n}$. Then by (\ref{Astfirstbound}) we have
\begin{align}\label{J1plusJ2inbboundlemma}
\|A_{s,t}\|_{L_p} &\lesssim |1-T|^{-1/2}\int_s^t \| \mathbb{E}^s(b_i(\bar{X}_r^n) - b_i(\bar{X}_{\kappa_n(r)}^n))\|_{L_p}\,dr \nonumber\\
&= |1-T|^{-1/2}\left(\int_s^\frac{k+4}{n} \| \mathbb{E}^s(b_i(\bar{X}_r^n) - b_i(\bar{X}_{\kappa_n(r)}^n))\|_{L_p}\,dr
+\int_\frac{k+4}{n}^t \| \mathbb{E}^s(b_i(\bar{X}_r^n) - b_i(\bar{X}_{\kappa_n(r)}^n))\|_{L_p}\,dr
\right) \nonumber\\
& =: |1-T|^{-1/2}(J_1 + J_2).
\end{align}
Note that using the trivial bound (\ref{integrandtrivialbound}) we have
\begin{align}\label{J1inbboundlemma}
J_1 &\lesssim \int_s^\frac{k+4}{n} n^{-\alpha/2} \,dr \nonumber\\
 &= n^{-\alpha/2}\left|(k+4)/n -s\right| \nonumber\\
 &= n^{-\alpha/2}\left|(k+4)/n -s\right|^{1/2 - \varepsilon}\left|(k+4)/n -s\right|^{1/2+\varepsilon} \nonumber\\
&\leq n^{-\alpha/2} \left |(k+4)/n -k/n\right|^{1/2 - \varepsilon} |t-s|^{1/2+\varepsilon} \nonumber\\
&\lesssim n^{-\frac{1+\alpha}{2}+\varepsilon}|t-s|^{1/2 + \varepsilon}.
\end{align}
Also, using the bound (\ref{integrandgoodbound}) we have
\begin{align}\label{J2inbboundlemma}
J_2 &\lesssim \int_\frac{k+4}{n}^t n^{-\frac{1+\alpha}{2} + (1+\alpha)\varepsilon} |\kappa_n(r) - (k+1)/n|^{-1/2} \,dr \nonumber\\
&\leq n^{-\frac{1+\alpha}{2} + (1+\alpha)\varepsilon} \int_\frac{k+4}{n}^t |r -(k+2)/n|^{-1/2}\,dr \nonumber\\
& = n^{-\frac{1+\alpha}{2} + (1+\alpha)\varepsilon}\left[2 |r -(k+2)/n|^{1/2}\right]_{r=\frac{k+4}{n}}^{r=t} \nonumber\\
&\lesssim n^{-\frac{1+\alpha}{2} + (1+\alpha)\varepsilon} |t-s|^{1/2} \nonumber\\
&\lesssim  n^{-\frac{1+\alpha}{2} + (2+\alpha)\varepsilon} |t-s|^{1/2+\varepsilon}
\end{align}
where the last inequality holds, because $n|t-s|\geq n|(k+4)/n - (k+1)/n| =3.$

By (\ref{J1plusJ2inbboundlemma}),(\ref{J1inbboundlemma}),(\ref{J2inbboundlemma}) we have that for $t\geq \frac{k+4}{n}$, the following bound holds:
$$\|A_{s,t}\|_{L_p} \lesssim |1-T|^{-1/2} n^{-\frac{1+\alpha}{2} + (2+\alpha)\varepsilon} |t-s|^{1/2+\varepsilon}.$$

Now we need to deal with the case $s\leq t \leq \frac{k+4}{n}$. We can obtain the same estimate simply using (\ref{Astfirstbound}) and the trivial bound (\ref{integrandtrivialbound}):

\begin{align*}
\|A_{s,t}\|_{L_p} &\lesssim|1-T|^{-1/2}\int_s^t \| \mathbb{E}^s(b(\bar{X}_r^n) - b(\bar{X}_{\kappa_n(r)}^n))\|_{L_p}\,dr.\\
&\lesssim |1-T|^{-1/2}\int_s^t n^{-\alpha/2} \,dr\\
& = |1-T|^{-1/2} n^{-\alpha/2} |t-s|\\
& =|1-T|^{-1/2} n^{-\alpha/2} |t-s|^{1/2 - \varepsilon}|t-s|^{1/2+\varepsilon}\\
&\leq |1-T|^{-1/2} n^{-\alpha/2} |(k+4)/n - k/n|^{1/2 - \varepsilon}|t-s|^{1/2+\varepsilon}\\
&\lesssim |1-T|^{-1/2} n^{-\frac{1+\alpha}{2} + \varepsilon}|t-s|^{1/2+\varepsilon}\\
&\lesssim |1-T|^{-1/2} n^{-\frac{1+\alpha}{2} + (2+\alpha)\varepsilon} |t-s|^{1/2+\varepsilon}.
\end{align*}
Thus condition (S1) holds with $\varepsilon_1 = \varepsilon$ and $C_1 = N|1-T|^{-1/2} n^{-\frac{1+\alpha}{2} + (2+\alpha)\varepsilon}$.

We will proceed by showing that condition (S2) holds. Note that

\begin{align*}
\mathbb{E}^s\delta A_{s,u,t} &= \mathbb{E}^s(A_{s,t} - A_{s,u} - A_{u,t})\\
&  = \mathbb{E}^s \Bigg(\mathbb{E}^s\int_s^t (b_i(\bar{X}_r^n) - b_i(\bar{X}_{\kappa_n(r)}^n))f_r(\bar{X}_s^n)dr - \mathbb{E}^s\int_s^u (b_i(\bar{X}_r^n) - b_i(\bar{X}_{\kappa_n(r)}^n))f_r(\bar{X}_s^n)\,dr\\
& \qquad\qquad - \mathbb{E}^u\int_u^t (b_i(\bar{X}_r^n) - b_i(\bar{X}_{\kappa_n(r)}^n))f_r(\bar{X}_u^n)\,dr\Bigg)\\
& = \int_u^t \mathbb{E}^s \mathbb{E}^u \left((b_i(\bar{X}_r^n) - b_i(\bar{X}_{\kappa_n(r)}^n))(f_r(\bar{X}_s^n) - f_r(\bar{X}_u^n))\right)\,dr\\
& = \int_u^t \mathbb{E}^s\left((f_r(\bar{X}_s^n) - f_r(\bar{X}_u^n)) \mathbb{E}^u(b_i(\bar{X}_r^n) - b(\bar{X}_{\kappa_n(r)}^n))\right)\,dr.
\end{align*}
Therefore
\begin{align}\label{Asut}
\|\mathbb{E}^s\delta A_{s,u,t}\|_{L_p} &\leq \int_u^t \|(f_r(\bar{X}_s^n) - f_r(\bar{X}_u^n)) \mathbb{E}^u(b_i(\bar{X}_r^n) - b_i(\bar{X}_{\kappa_n(r)}^n))\|_{L_p}\,dr \nonumber\\
&\leq \int_u^t \|f_r(\bar{X}_s^n) - f_r(\bar{X}_u^n)\|_{L_{2p}} \|\mathbb{E}^u(b_i(\bar{X}_r^n) - b_i(\bar{X}_{\kappa_n(r)}^n))\|_{L_{2p}}\,dr \nonumber\\
&\leq \sup_{r \in [u,t]}\|f_r\|_{C^1}\|\bar{X}_s^n - \bar{X}_u^n\|_{L_{2p}} \int_u^t  \|\mathbb{E}^u(b_i(\bar{X}_r^n) - b_i(\bar{X}_{\kappa_n(r)}^n))\|_{L_{2p}}\,dr.
\end{align}
We will now bound each parts of this expression. Using our previous result that we obtained while bounding $\|A_{s,t}\|$, we have
\begin{align}\label{Asut1}
\int_u^t\|\mathbb{E}^u(b_i(\bar{X}_r^n) - b_i(\bar{X}_{\kappa_n(r)}^n))\|_{L_{2p}}\,dr &\lesssim |t-u|^{1/2 + \varepsilon} n^{-\frac{1+\alpha}{2} + (2+\alpha)\varepsilon} \nonumber\\
 &\lesssim |t-s|^{1/2 + \varepsilon} n^{-\frac{1+\alpha}{2} + (2+\alpha)\varepsilon}.
\end{align}
Moreover, using the Schauder-estimate Corollary \ref{uestimates},
\begin{equation}\label{Asut2}
\sup_{r \in [u,t]}\|f_r\|_{C^1} \lesssim \sup_{r\in [u,T]}|1-r|^{-1}\|g\|_{C^0} \lesssim |1-T|^{-1}.
\end{equation}
It is easy to see that
\begin{equs}\label{Asut3}
\|\bar{X}_s^n - \bar{X}_u^n\|_{L_{2p}} \lesssim |t-s|^{1/2}.
\end{equs}
Now using the bounds (\ref{Asut1}),(\ref{Asut2}),(\ref{Asut3}), we can bound (\ref{Asut}) as follows:
\begin{align*}
\|\mathbb{E}^s\delta A_{s,u,t}\|_{L_p}&\lesssim |1-T|^{-1} |t-s|^{1/2} |t-s|^{1/2 + \varepsilon} n^{-\frac{1+\alpha}{2} + (2+\alpha)\varepsilon}\\
& \lesssim |1-T|^{-1}n^{-\frac{1+\alpha}{2} + (2+\alpha)\varepsilon} |t-s|^{1+\varepsilon}.
\end{align*}
Hence (S2) holds with $\varepsilon_2 = \varepsilon$ and $C_2 =N |1-T|^{-1}n^{-\frac{1+\alpha}{2} + (2+\alpha)\varepsilon}.$

Let $$\bar{\mathcal{A}}_t : =\int_0^t (b_i(\bar{X}_r^n) - b_i(\bar{X}_{\kappa_n(r)}^n))f_r(\bar{X}_r^n)\,dr.$$
We will show that conditions (S3) and (S4) hold with $\bar{\mathcal{A}}$ in place of $\mathcal{A}$.
We can see that condition (S3) holds, because for any $(s,t)\in [0,T]_{\leq}$ we have
\begin{align*}
\|\bar{\mathcal{A}}_t - \bar{\mathcal{A}}_s - A_{s,t}\|_{L_p} &= \Big\| \int_s^t(b_i(\bar{X}_r^n) - b_i(\bar{X}_{\kappa_n(r)}^n))f_r(\bar{X}_r^n)\,dr\\
&\qquad\qquad - \int_s^t \mathbb{E}^s((b_i(\bar{X}_r^n) - b_i(\bar{X}_{\kappa_n(r)}^n))f_r(\bar{X}_s^n))\,dr\Big\|_{L_p}\\
&\lesssim \|b\|_{C^0}\sup_{r \in [s,t]}\|f_r\|_{C^0}|t-s|\\
&\lesssim (1-T)^{-1/2}|t-s|\\
&\lesssim (1-T)^{-1/2}|t-s|^{1/2 + \varepsilon}.
\end{align*}
Furthermore, condition (S4) holds, because for any $(s,t)\in[0,T]_{\leq}$,
\begin{align*}
\|\mathbb{E}^s(\bar{\mathcal{A}}_{t}- \bar{\mathcal{A}}_s - A_{s,t} )\|_{L_p} &= \Bigg\|\mathbb{E}^s \Bigg(\int_s^t (b_i(\bar{X}_r^n) - b_i(\bar{X}_{\kappa_n(r)}^n)f_r(\bar{X}_r^n)\,dr \\
&\qquad\qquad- \mathbb{E}^s \int_s^t (b_i(\bar{X}_r^n) - b_i(\bar{X}_{\kappa_n(r)}^n)f_r(\bar{X}_s^n)\,dr\Bigg)\Bigg\|_{L_p}\\
&\leq \int_s^t \| (b_i(\bar{X}_r^n) - b_i(\bar{X}_{\kappa_n(r)}^n)(f_r(\bar{X}_r^n) - f_r(\bar{X}_s^n))\|_{L_p}\,dr\\
&\lesssim \|b\|_{C^0} \int_s^t \|f_r\|_ {C^1}\| \bar{X}_r^n - \bar{X}_{s}^n\|_{L_p}\,dr\\
&\lesssim (1-T)^{-1} \int_s^t|r-s|^{1/2}\,dr\\
&\lesssim (1-T)^{-1}|t-s|^{3/2}\\
&\lesssim (1-T)^{-1} |t-s|^{1+\varepsilon},
\end{align*}
where we used that $\|\bar{X}_r^n - \bar{X}_s^b\| \lesssim |r-s|^{1/2}$, and Corollary \ref{uestimates}.

Since conditions (S1) - (S4) of Lemma \ref{sewinglemma} are satisfied, by uniqueness of $\mathcal{A}$ we have that
$$\mathcal{A} = \bar{\mathcal{A}}$$
and the bound we aimed to show holds for any smooth $b \in C^\alpha$. The smoothness assumption then can be removed by constructing a sequence of smooth approximations of $b \in C^\alpha$ via mollification and then passing to the limit. This finishes the proof.
\end{proof}

We will now proceed by showing that the result of Lemma \ref{bquadrature} still holds if we consider $X^n$ in place of the driftless scheme $\bar{X}^n$.

\begin{lemma}\label{bquadraturegeneral}
Let $p\geq 1$, $\eps>0$, and let  $X^n$ be given by (\ref{ourscheme}). Then  for all $n\in \mathbb{N}, (s,t)\in[0,T]_{\leq}$ we have
\begin{multline*}\left\|\int_s^t (b_i({X}_r^n)  - b_i({X}_{\kappa_n(r)}^n))f_r({X}_r^n)\,dr\right\|_{L_p} \leq  N n^{-\frac{1+\alpha}{2} + \varepsilon}\Big(|1-T|^{-1/2}|t-s|^{1/2+\varepsilon} \\
+ |1-T|^{-1}|t-s|^{1+\varepsilon}\Big)
\end{multline*}
for some constant $N$ depending only on  $d, \alpha,p, \varepsilon, \lambda, \|b\|_{C^\alpha}, \|\sigma\|_{C^2}, \|g\|_{C^0}$.
\end{lemma}
\begin{proof} The proof is a standard argument via Girsanov's theorem. By Hölder's inequality it suffices to show that the bound holds sufficiently large $p$. Let
$$\rho := \exp\left(-\int_0^1 (\sigma^{-1} b)(X_{\kappa_n(t)}^n)\,dW_r - \frac{1}{2}\int_0^1 |(\sigma^{-1} b)(X_{\kappa_n(t)}^n)|^2 \,dt\right).$$ By the assumptions on $b$ and $\sigma$, we have that $\sigma^{-1} b$ is bounded, hence $\mathbb{E}\rho =1$, thus Girsanov's theorem applies. Construct a new probability measure $\tilde{\mathbb{P}}$ by
$$d\tilde{\mathbb{P}} = \rho \,d \mathbb{P}.$$
Then under $\tilde{\mathbb{P}}$, the process
$$\tilde{W}_t := W_t + \int_0^t (\sigma^{-1}b)(X_{\kappa_n(r)}^n)\,dr, \quad t \in [0,1]$$ 
is an $\mathbb{F}$-Wiener process. Notice that
$$dX_t^n = \sigma(X_{\kappa_n(t)}^n)\,d\tilde{W}_t,$$
hence the distribution of $X^n$ under $\tilde{\mathbb{P}}$ coincides with the distribution of $\bar{X}^n$ under $\mathbb{P}$. For any continuous process $Z$, let us denote $\xi(Z) = \int_s^t (b(Z_r) - b(Z_{\kappa_n(r)}))f_r(Z_r)dr$. Then
\begin{align*}
\mathbb{E}|\xi(X^n)|^p  &= \tilde{\mathbb{E}}\left(|\xi(X^n)|^p \rho^{-1}\right)\\
&\leq \left(\tilde{\mathbb{E}}|\xi(X^n)|^{2p}\right)^{1/2}(\tilde{\mathbb{E}}\rho^{-2})^{1/2}\\
& = \left(\mathbb{E}|\xi(\bar{X}^n)|^{2p}\right)^{1/2}(\mathbb{E}\rho^{-1})^{1/2}.
\end{align*}
From the assumptions on $b$ and $\sigma$ it follows that $\rho$ has finite moments of any order. Hence $\mathbb{E}|\xi(X^n)|^p \lesssim \left(\mathbb{E}|\xi(\bar{X}^n)|^{2p}\right)^{1/2}$, and thus
$$\|\xi(X^n)\|_{L_p} \lesssim \| \xi(\bar{X})\|_{L_{2p}}.$$
Now the bound we aimed to show follows from Lemma \ref{bquadrature}.
\end{proof}

Finally, in in order to bound  the expression (\ref{bexpression1}), we need to modify Lemma \ref{bquadraturegeneral} by extending the domain of integration from $[s,t]$ (where $0\leq s\leq t\leq T<0$) to $[0,1]$. The difficulty of this comes from the blowup of the bound in Lemma \ref{bquadraturegeneral} as $T\to 1$. We will thus need to ``trade'' the $|t-s|$-dependence of the bound to extend the domain of integration as desired. This can be done using dyadic points.

\begin{corollary}\label{boundonbexpression}
 Let  $X^n$ be given by (\ref{ourscheme}) and let $\varepsilon >0$, $p \geq 1$. Then, we have
\begin{align*}
\left\|\int_0^1 (b_i(X_r^n) - b_i(X_{\kappa_n(r)}^n))f(X_r^n)\,dr \right\|_{L_p} \leq N n^{-\frac{1+\alpha}{2} + \varepsilon}
\end{align*}
for some constant $N$ depending only on $d,\alpha,p, \varepsilon, \lambda, \|b\|_{C^\alpha}, \|\sigma\|_{C^2}, \|g\|_{C^0}$.
\end{corollary}

\begin{proof}

For simplicity, let us use the notation
$Y_r^n : = (b_i(X_r^n) - b_i(X_{\kappa_n(r)}^n))f(X_r^n)$.

Note, that
\begin{align}\label{A1plusA2}
\left\| \int_0^{1-2^{-m}} Y_r^n \,dr\right \|_{L_p} &\leq \left\|\int_0^{1/2} Y_r^n \,dr\right\|_{L_p} + \left\| \int_{1/2}^{1-2^{-m}}Y_r^n dr\right\|_{L_p} \nonumber\\
& =: A_1 + A_2(m).
\end{align}

By Lemma \ref{bquadraturegeneral} we have that
\begin{align*}
A_1  &= \left\|\int_0^{1/2} Y_r^n \,dr \right\|_{L_p} \lesssim n^{-\frac{1+\alpha}{2}}.
\end{align*}

To bound $A_2(m)$, note the elementary identity
$$\sum_{k=1}^{m} 2^{-k} = 1 - 2^{-m}$$
and that therefore the collection of intervals $\{(\sum_{k=1}^j 2^{-k}, \sum_{k=1}^{j+1} 2^{-k}] : j \in \{1,\dots, m-1\}\}$ forms a partition of the interval $(1/2,1 - 2^{-m}]$. Using this, and Lemma \ref{bquadraturegeneral}, we have
\begin{align*}
A_2(m) &= \left\| \int_{1/2}^{1-2^{-m}}Y_r^n \,dr\right\|_{L_p}\\
& = \left\| \sum_{j=1}^{m-1} \int_{\sum_{k=1}^j 2^{-k}}^{\sum_{k=1}^{j+1} 2^{-k}} Y_r^n \,dr \right\|_{L_p}\\
& \leq\sum_{j=1}^{m-1} \left\| \int_{\sum_{k=1}^j 2^{-k}}^{\sum_{k=1}^{j+1} 2^{-k}} Y_r^n \,dr \right\|_{L_p}\\
&\lesssim \sum_{j=1}^{m-1} n^{-\frac{1+\alpha}{2}+{\varepsilon}}\Bigg(\left(1 - \sum_{k=1}^{j+1} \left(\frac{1}{2}\right)^k\right)^{-1/2}\left|\left(\frac{1}{2}\right)^{j+1}\right|^{1/2+\varepsilon}
+ \left(1 - \sum_{k=1}^{j+1} \left(\frac{1}{2}\right)^k\right)^{-1}\left|\left(\frac{1}{2}\right)^{j+1}\right|^{1+\varepsilon}\Bigg)\\
& \lesssim  n^{-\frac{1+\alpha}{2}+{\varepsilon}}  \sum_{j=1}^{m-1}\Bigg(\left(2^{-(j+1)}\right)^{-1/2} 2^{-(j+1)(1/2+\varepsilon)}
+ \left(2^{-(j+1)}\right)^{-1} 2^{ -(j+1)(1+\varepsilon)}\Bigg)\\
& \lesssim n^{-\frac{1+\alpha}{2}+{\varepsilon}}  \sum_{j=1}^{m-1}
\Bigg(2^{\frac{j+1}{2} - \frac{j+1}{2} - \varepsilon(j+1)} + 2^{(j+1) - (j+1) - \varepsilon(j+1)}\Bigg)\\
&\lesssim n^{-\frac{1+\alpha}{2}+{\varepsilon}} 2 \sum_{j=1}^{m-1} 2^{-\varepsilon(j+1)}\\
& \lesssim n^{-\frac{1+\alpha}{2}+{\varepsilon}} 2^{1-\varepsilon}\sum_{j=1}^{m-1} (2^{-\varepsilon})^j
\end{align*}
Therefore
\begin{align*}
\lim_{m\to \infty} A(m) &\lesssim n^{-\frac{1+\alpha}{2}+{\varepsilon}} 2^{1-\varepsilon}\sum_{j=1}^{\infty} (2^{-\varepsilon})^j\\
&\lesssim n^{-\frac{1+\alpha}{2}+{\varepsilon}} 2^{1-\varepsilon}\frac{1}{1 - 2^{-\varepsilon}}\\
& \lesssim \frac{2}{2^\varepsilon -1}n^{-\frac{1+\alpha}{2}+{\varepsilon}} 
\end{align*}
Now passing to the limit $m \to \infty$ in (\ref{A1plusA2}) and using our bounds on $A_1, A_2$ gives the result.
\end{proof}

Corollary \ref{boundonbexpression} above will be used to bound (\ref{bexpression0}). We will now prove that a similar bound holds on $(\ref{sigmaexpression0})$. For notational convenience, we will replace $(\sigma\sigma^T)_{ij}$ in $(\ref{sigmaexpression0})$ with a function $h:\mathbb{R}^d \to \mathbb{R}$ of identical regularity.
We moreover fix $i,j \in \{1,\dots,d\}$ and denote 
$$f'_t := \partial_{x_i x_j} u(t,\cdot).$$

\begin{lemma}\label{sameboundforotherterm}
Let $X^n$ given by $(\ref{ourscheme})$, let $h \in C^2$ and $\varepsilon>0$. Then, for all $n \in \mathbb{N}$ we have
\begin{align*}
\left|\mathbb{E}\int_0^1 (h(X_r^n) - h(X_{\kappa_n(r)}^n))f'_r(X_r^n) \,dr \right| \leq N n^{-\frac{1+\alpha}{2}+ \varepsilon}
\end{align*}
for some constant $N$ depending only on $d, \alpha, \lambda, \|b\|_{C^\alpha},\|\sigma\|_{C^2} \|h\|_{C^2}, \|g\|_{C^\alpha}$.
\end{lemma}
\begin{proof}
Let $\varepsilon>0$. Note, that we have 
\begin{equation}\label{Efirstboundh}
\left|\mathbb{E}\int_0^1 (h(X_r^n) - h(X_{\kappa_n(r)}^n))f'_r(X_r^n) \,dr \right| \leq \int_0^1 | \mathbb{E}(h(X_r^n) - h(X_{\kappa_n(r)}^n))f'_r(X_r^n)|\, dr.
\end{equation}
We will now establish bounds on the integrand. Note that denoting $\delta = r - \kappa_n(r) \leq n^{-1}$, we have
\begin{align}\label{EissupofJ}
&|\mathbb{E}((h(X_r^n) - h(X_{\kappa_n(r)}^n))f'_r(X_r^n))| = \nonumber\\
&=\Bigg|\mathbb{E}\mathbb{E}^{\kappa_n(r)} \bigg[ \bigg(h\left(X_{\kappa_n(r)}^n + b(X_{\kappa_n(r)}^n)\delta + \sigma(X_{\kappa_n(r)}^n)(W_r - W_{\kappa_n(r)})\right)- h\left(X_{\kappa_n(r)}^n\right)\bigg) \nonumber\\
&\qquad\qquad\qquad f'_r\left(X_{\kappa_n(r)}^n + b(X_{\kappa_n(r)}^n)\delta + \sigma(X_{\kappa_n(r)}^n)(W_r - W_{\kappa_n(r)})\right)\bigg]\Bigg| \nonumber\\
& = \Bigg| \mathbb{E}\Bigg( \mathbb{E} \bigg[\left(h\left(x+ b(x)\delta + \sigma(x)(W_r  - W_{\kappa_n(r)})\right) - h(x)\right) \nonumber\\
&\qquad\qquad\qquad f'_r\left(x+ b(x)\delta + \sigma(x)(W_r  - W_{\kappa_n(r)})\right) \bigg]\bigg\vert_{x = X_{\kappa_n(r)}^n}\Bigg) \Bigg| \nonumber\\
&\leq \sup_{x \in \mathbb{R}^d}\left| \int_{\mathbb{R}} (h(x+b(x)\delta + y) - h(x))f'_r(x+b(x)\delta + y) p_{\sigma\sigma^T(x)\delta}(y) \,dy \right|.
\end{align}
We decompose the integral in the above expression in the following way:
\begin{align}\label{J1plusJ2}
J:=&\int_{\mathbb{R}} (h(x+b(x)\delta + y) - h(x))f'_r(x+b(x)\delta + y) p_{\sigma\sigma^T(x)\delta}(y) \,dy \nonumber\\
& = \int_{\mathbb{R}^d}p_{\sigma\sigma^T(x)\delta}(y)f'_r(x + b(x)\delta + y) (h(x+b(x)\delta+y)- h(x+y)) \,dy \nonumber\\
&\qquad + \int_{\mathbb{R}^d}p_{\sigma\sigma^T(x)\delta}(y)f'_r(x + b(x)\delta + y) (h(x+y) -h(x)) \,dy \nonumber\\
& =: J_1 + J_2.
\end{align}
The integral $J_1$ is easier to bound, as using the regularity of $h$ we can get a factor of $\delta$ as follows:
\begin{equs}\label{J1bound}
|J_1| & = \left|\int_{\mathbb{R}^d}p_{\sigma\sigma^T(x)\delta}(y)f'_r(x + b(x)\delta + y) (h(x+b(x)\delta+y)- h(x+y)) \,dy \right|\\
&\leq \int_{\mathbb{R}^d} p_{\sigma\sigma^T(x)\delta}(y) \|f'_r\|_{C^0}\|h\|_{C^1}\|b\|_{C^0}\delta \,dy\\
&\lesssim \|u(r,\cdot)\|_{C^2} n^{-1}\\
&\lesssim |1-r|^{-1+\alpha/2} n^{-1}
\end{equs}
where the last inequality holds by Corollary \ref{uestimates}. We will now bound $J_2$. To this end, note that
\begin{equs}
J_2 &= \int_{\mathbb{R}^d}p_{\sigma\sigma^T(x)\delta}(y) f'_r(x+b(x)\delta+y)(h(x+y) -h(x))\,dy \nonumber\\
& = \sum_{k=1}^d \int_{\mathbb{R}^d}p_{\sigma\sigma^T(x)\delta}(y)f'_r(x+b(x)\delta+y) \int_0^1 \partial_{x_k} h(x+\theta y ) y_k \,d\theta \,dy.
\end{equs}
Note that if we removed the $y$-dependence of $f'_r$ and of $\partial_{x_k} h$. then the integral would be zero, since $\int_{\mathbb{R}^d} y_k  p_{\sigma\sigma^T(x)\delta}(y) \,dy =0$. Hence
\begin{align}\label{J21plusJ22}
|J_2| &= \left|\sum_{k=1}^d \int_{\mathbb{R}^d} p_{\sigma\sigma^T(x)\delta}(y) y_k \int_0^1 \Big( f'_r(x+b(x)\delta+y)\partial_{x_k} h(x+\theta y) - f'_r(x+b(x)\delta)\partial_{x_k}h(x)\Big) \,d\theta \,dy\right|\nonumber\\
 &\lesssim \sup_{k}\left| J_{2,1}^k+J_{2,2}^k\right|
\end{align}
for
\begin{equ}
J_{2,1}^k =  \int_{\mathbb{R}^d} p_{\sigma\sigma^T(x)\delta}(y) y_k \int_0^1 f'_r(x+b(x)\delta+y)\Big( \partial_{x_k}h(x+\theta y) - \partial_{x_k}h(x)\Big) \,d\theta \,dy
\end{equ}
and
\begin{equ}
J_{2,2}^k = \int_{\mathbb{R}^d} p_{\sigma\sigma^T(x)\delta}(y) y_k \int_0^1 \partial_{x_k}h(x)\Big(f'_r(x+b(x)\delta+y) - f'_r(x+b(x)\delta)\Big) \,d\theta \,dy.
\end{equ}
We bound $J_{2,1}^k$ using the regularity of $h$:
\begin{align}\label{J21bound}
|J_{2,1}^k| &\leq \int_{\mathbb{R}^d} p_{\sigma\sigma^T(x)\delta}(y)|y| \|f'_r\|_{C^0}\|h\|_{C^2}|y| \,dy \nonumber\\
&\lesssim  \|u(r,\cdot)\|_{C^2} \int_{\mathbb{R}} |y|^2 p_{\sigma\sigma^T(x)\delta}(y)\,dy \nonumber\\
&\lesssim |1-r|^{-1+\alpha/2}\|\sigma\|_{C^0}^2 n^{-1} \nonumber\\
&\lesssim |1 - r|^{-1+\alpha/2} n^{-1}.
\end{align}
We moreover bound $J_{2,2}^k$ using the $\alpha$-Hölder-regularity of $f'$. This will increase the blowup in time:
\begin{align}\label{J22bound}
|J_{2,2}^k| &\lesssim \int_{\mathbb{R}^d} p_{\sigma\sigma^T(x)\delta}(y)|y|\|h\|_{C^1} \|f'\|_{C^\alpha}|y|^\alpha \,dy \nonumber\\
&\lesssim \|u(r,\cdot)\|_{C^{2+\alpha}} n^{-\frac{1+\alpha}{2}} \nonumber\\
&\lesssim |1 -r|^{-1}n^{-\frac{1+\alpha}{2}}.
\end{align}
By (\ref{J1plusJ2}),(\ref{J1bound}),(\ref{J21plusJ22}),(\ref{J21bound}),(\ref{J22bound}) we conclude that
$J \lesssim |1-r|^{-1}n^{-\frac{1+\alpha}{2}}$. Hence by (\ref{EissupofJ}) we have
\begin{equation}\label{nontrivialEbound}
|\mathbb{E}(h(X_r^n) - h(X_{\kappa_n(r)}^n))f'_r(X_r^n))| \lesssim |1-r|^{-1} n^{-\frac{1+\alpha}{2}}.
\end{equation}
The right hand side converges at the desired rate, however it cannot be integrated on $[0,1]$ with respect to $r$. Therefore we will also need to find an other bound with slower blowup at the terminal time.
To this end, note that $J_{2,2}$ can be also bounded as follows:
\begin{align}\label{trivialJ22bound}
|J_{2,2}| &\leq \int_{\mathbb{R}} p_{\sigma^2(x)\delta}(y) |y| \|h\|_{C^1}\|f'_r\|_{C^0} \,dy \nonumber\\
&\lesssim \| u(r,\cdot)\|_{C^2} \int_{\mathbb{R}} |y| p_{\sigma^2(x)\delta}(y) \,dy \nonumber\\
&\lesssim |1-r|^{-1+\alpha/2} n^{-1/2}.
\end{align}
So by (\ref{J1plusJ2}),(\ref{J1bound}),(\ref{J21plusJ22}),(\ref{J21bound}),(\ref{trivialJ22bound}) we have
\begin{align}\label{trivialEbound}
|\mathbb{E}(h(X_r^n) - h(X_{\kappa_n(r)}^n))f'_r(X_r^n))| \lesssim |1-r|^{-1+\alpha/2} n^{-1/2}
\end{align}
We are finally in the position to bound the expression we originally aimed to. By (\ref{Efirstboundh}),(\ref{nontrivialEbound}),(\ref{trivialEbound}) we have that for $a\in(0,1)$,
\begin{align*}
\left|\mathbb{E}\int_0^1 (h(X_r^n) - h(X_{\kappa_n(r)}^n))f'_r(X_r^n) \,dr \right| &\lesssim \int_0^a (1-r)^{-1}n^{-\frac{1+\alpha}{2}}\,dr + \int_a^1 (1-r)^{-1+\alpha/2} n^{-1/2}\, dr\\
& \lesssim n^{-\frac{1+\alpha}{2}}\left[ -\log(1-r)\right]_{r=0}^{r =a} + n^{-1/2}\left[ - (1-r)^{\alpha/2}\right]_{r=a}^1,
\end{align*}
thus choosing $a = 1 - \frac{1}{n}$, we have
\begin{align*}
\left|\mathbb{E}\int_0^1 (h(X_r^n) - h(X_{\kappa_n(r)}^n))f'_r(X_r^n) \,dr \right| &\lesssim n^{-\frac{1+\alpha}{2}}\left( - \log\left(\frac{1}{n}\right) +0\right)+ n^{-1/2}\left(0+ \left(\frac{1}{n}\right)^{\alpha/2}\right)\\
&\lesssim \log(n) n^{-\frac{1+\alpha}{2}}\\
&\lesssim n^{-\frac{1+\alpha}{2} + \varepsilon}
\end{align*}
as required.
\end{proof}

\section{Proof of the main result}\label{proofofmainresultsection}

Having established the quadrature estimates of section \ref{quadratureestimatessection}, we are now in the position to prove the main result.
\begin{proof}[Proof of Theorem \ref{mainresult}]

Let $u$ be the unique bounded solution of (\ref{ourpde}).
By the Feynmann-Kac formula, we have
$$u(0,x_0) = \mathbb{E}g(X_1),$$
furthermore, by the terminal condition we have
$$u(1,X_1^n) = g(X_1^n).$$
Hence the weak error can be expressed as
\begin{align}\label{talaytrick}
d_g(X,X^n)&:=\left|\mathbb{E}g(X_1^n) - \mathbb{E}g(X_1)\right|& \nonumber\\
&\,\,= \left|\mathbb{E}\left( u(1,X_1^n) - u(0,x_0)\right)\right| \nonumber\\
& \,\,= \left|\mathbb{E}\left(u(1,X_1^n) - u(0,X_0^n)\right)\right|.
\end{align}
Applying Itô's lemma gives

\begin{multline}
d_g(X,X^n)=\Bigg|\mathbb{E}\Bigg(\int_0^1\Big( \partial_t u(r,X_r^n) + \sum_{i=1}^d \partial_{x_i}u(r,X_r^n)b_i(X_{\kappa_n(r)}^n) + \\
\quad\frac{1}{2}\sum_{i,j =1}^d \partial_{x_i x_j}u(r,X_r^n)\sum_{k=1}^{d}\sigma_{ik}(X_{\kappa_n(r)}^n)\sigma_{jk}(X_{\kappa_n(r)}^n)\Big)\,dr\Bigg)\\
+\mathbb{E}\Bigg(\int_0^1 \sum_{i=1}^d\sum_{j=1}^{d}\partial_{x_i}u(r,X_r^n)\sigma_{ij}(X_{\kappa_n(r)}^n)\,dW_r^j\Bigg)\Bigg|.
\end{multline}

Note that since $g$ has positive Hölder-regularity, by Lemma \ref{uestimates} the integrand $\partial_{x_i} u(t,X_t^n)\sigma_{ij}(X_{\kappa_n(t)}^n)$ is square-integrable on $\Omega \times [0,1]$, and thus in the above expression the stochastic integral term has zero expectation.

Hence we are left with
$$d_g(X,X^n)= \left|\mathbb{E}\int_0^1\left( \partial_t u (r,X_r^n) + \bar{L} u \left(r, X_r^n, X_{\kappa_n(r)}^n\right) \right)\,dr\right|,$$
where the operator $\bar{L}$ is given by
$$\bar{L}\phi(t,x, \bar{x}) : = \sum_{i=1}^d b(\bar{x})\partial_{x_i} \phi(x) + \frac{1}{2}\sum_{i,j=1}^d (\sigma\sigma^T)_{ij}(\bar{x})\partial_{x_i x_j}\phi(x).$$
Recall that by our parabolic PDE (\ref{ourpde}), we have $\partial_t u =- Lu$. Hence the weak error can be written as
$$d_g(X,X^n)= \left| \mathbb{E}\int_0^1\left(\bar{L}u\left(r,X_r^n, X_{\kappa_n(r)}\right) - Lu\left(r,X_r^n\right)\right)\,dr\right|.$$
Writing out the operators explicitly gives
\begin{multline}
d_g(X,X^n)= \Bigg|\mathbb{E}\Bigg(\int_0^1 \Big(\sum_{i=1}^d(b_i(X_{\kappa_n(r)}^n) - b_i(X_r^n) \partial_{x_i}u(r,X_r^n)\\
+ \frac{1}{2}\sum_{i,j =1}^d((\sigma\sigma^T)_{ij}(X_{\kappa_n(r)}^n) - (\sigma\sigma^T)(X_r^n))\partial_{x_i x_j} u(r,X_r^n)\Big)\,dr\Bigg)\Bigg|
\end{multline}
So by the triangle-inequality we can see that it suffices to show that our bound holds for
\begin{equation}\label{bexpression}
\left|\mathbb{E}\int_0^1 \left(\left(b_i(X_{\kappa_n(r)}^n) - b_i(X_r^n)\right)\partial_{x_i}u(r,X_r^n)\right)\,dr\right|
\end{equation}
and for
\begin{equation}\label{sigmaexpression}
\left|\mathbb{E}\int_0^1 \left(\left(\sigma\sigma^T)_{ij}(X_{\kappa_n(r)}^n) - (\sigma\sigma^T)_{ij}(X_r^n)\right)\partial_{x_i x_j}u(r,X_r^n)\right)\,dr\right|.
\end{equation}
for all $i,j$ indices.

The desired bound now follows from Corollary \ref{boundonbexpression} and Lemma \ref{sameboundforotherterm}.
\end{proof}

\begin{remark}
If (\ref{sigmaexpression}) is zero, then we do not need Lemma \ref{sameboundforotherterm} which would require $g$ to be $C^\alpha$. Thus for the additive noise case, Theorem \ref{mainresult} holds when $g$ has any positive Hölder-regularity, i.e. for any $g \in C^\zeta$ for any $\zeta>0$.
\end{remark}

\section{Acknowledgements}
The author would like to thank Konstantinos Dareiotis for his advice.

\bibliographystyle{Martin}
\bibliography{references}
\end{document}